\documentclass[centertags,11pt]{amsart}
\usepackage[top=4.3cm,bottom=4.3cm,left=3.3cm,right=3.3cm]{geometry}               
\usepackage{enumerate}
\usepackage{amsmath}
\usepackage{amssymb,latexsym}
\usepackage{amsthm}
\usepackage{color}
\usepackage{setspace}
\usepackage{graphicx}
\usepackage{hyperref}
\usepackage{amscd}

\DeclareMathOperator*{\Gal}{Gal}
\DeclareMathOperator*{\Hom}{Hom}

\theoremstyle{definition}
\newtheorem{definition}{Definition}[section]
\newtheorem{remark}[definition]{Remark}

\newtheorem{theorem}[definition]{Theorem}
\newtheorem{proposition}[definition]{Proposition}
\newtheorem{lemma}[definition]{Lemma}
\newtheorem{corollary}[definition]{Corollary}

\newtheorem*{theorem*}{Theorem}

\begin{document}

\title{Algebraic constructions of complete $m$-arcs}


\author[D. Bartoli]{Daniele Bartoli}
\address{
Dipartimento di Matematica e Informatica, Universit\`a degli Studi di Perugia}
	\email{daniele.bartoli@unipg.it}
\author[G. Micheli]{Giacomo Micheli}
\address{Department of Mathematics, University of South Florida, Tampa, FL 33620 (US)}
	\email{gmicheli@usf.edu}

\begin{abstract}
Let $m$ be a positive integer, $q$ be a prime power, and $\mathrm{PG}(2,q)$ be the projective plane over the finite field $\mathbb F_q$. Finding complete $m$-arcs in $\mathrm{PG}(2,q)$ of size less than $q$ is a classical problem in finite geometry.
In this paper we give a complete answer to this problem when $q$ is relatively large compared with $m$, explicitly constructing the smallest $m$-arcs in the literature so far for any $m\geq 8$. For any fixed $m$, our arcs $\mathcal A_{q,m}$ satisfy $|\mathcal A_{q,m}|-q\rightarrow -\infty$ as $q$ grows.
To produce such $m$-arcs, we develop a Galois theoretical machinery that allows the transfer of geometric information of points external to the arc, to arithmetic one, which in turn allows to prove the $m$-completeness of the arc.
\end{abstract}

\maketitle

\section{Introduction}
Let $p$ be a prime, $h$ be a positive integer, and $q=p^a$. Denote by $\mathrm{PG}(2,q)$ the projective Galois plane over the field $\mathbb{F}_q$ with $q$ elements. The symbol $\overline{\mathbb{F}_q}$ denotes the algebraic closure of $\mathbb{F}_q$. The corresponding affine plane of order $q$ is denoted by $\mathrm{AG}(2,q)=\mathrm{PG}(2,q)\setminus \ell_{\infty}$, where $\ell_{\infty}$ is the line at infinity.  A $(k, m)$-arc $\mathcal{A}$ in $\mathrm{PG}(2,q)$ is a set of $k$ points no $m + 1$ of which are collinear and such that there exist $m$ collinear points.  The arc $\mathcal{A}$ is said \emph{complete} if it is maximal with respect to the set theoretical inclusion. A general introduction to $(k, m)$-arcs can be found in the monograph \cite[Chapter 12]{10}, as well as in the survey paper \cite[Section 5]{12}. 

Arcs in projective planes have interesting connection with Coding Theory too: complete $(k, m)$-arcs correspond to linear $[k, 3, k-m]_q$-codes which cannot be extended to a code with the same dimension and minimum distance.

As a definition, we say that a set of points $\mathcal{A}\subset \mathrm{PG}(2,q)$ $m$-covers (or simply covers, if the context is clear) an external point $P\in \mathrm{PG}(2,q)\setminus \mathcal{A}$ if there exist $P_1,\ldots,P_m$ pairwise distinct points of $\mathcal{A}$ such that $P_1,\ldots,P_m,P$ are collinear. This is equivalent to say that there exists an $m$-secant to $\mathcal{A}$ passing through $P$. 

Concerning the smallest interesting case $m=2$, the theory is well developed and quite rich of constructions; see e.g. \cite{1,2,3,9,12,18,19} and the references therein, as well as \cite[Chapters 8-10]{10}.  The notion of arcs (with $m=2$) has been introduced by Segre \cite{Segre1,Segre2} in the late 1950's. He was also interested  to know how many points a complete arc in a projective plane of order $q$ can have; see \cite{Segre1}. Much is known about the maximum sizes: the upper bound is $q+1$ if $q$ is odd and $q+2$ if $q$ is even and both are actually attained. In the case of equality, in odd characteristic, Segre himself proved that the points of the arc form a conic, whereas in the even characteristic case the classification of such arcs is not completed. The second largest cardinality of a  $(k, 2)$-arc has been widely investigated too. Much less is known about the smallest possible size. Probabilistic constructions of small complete arcs have been obtained in the seminal paper \cite{KV}, where the authors prove the existence of a complete arc of size $O(\sqrt{q}\log^c q)$ in a projective plane (not necessarily $\mathrm{PG}(2,q)$) of order $q>M$, for some positive constants $c,M$ independent on $q$. All the other existence results obtained so far in the literature  either come from computer searches, or follow a suggestion of Segre and Lombardo-Radice; \cite{15,16}. Their idea was to choose a special ``small" set $\mathcal{A}$ of points from a fixed algebraic curve  and then to complete it   adding a ``few" more extra points. This machinery has been intensively used in the last decades: usually, the collinearity of two points $Q$ and $R$ of $\mathcal{A}$ and a point $P\notin \mathcal{A}$ is translated into the existence of a suitable $\mathbb{F}_q$-rational point of a certain  $\mathbb{F}_q$-rational curve $\mathcal{C}_{\mathcal{A}}$ related to $\mathcal{A}$; see for instance \cite{DiComite,DiComite2,Zirilli}. In this direction, investigation of the curve $\mathcal{C}_{\mathcal{A}}$ (in particular the existence of a suitable $\mathbb{F}_q$-rational absolutely irreducible component)  and Hasse-Weil Theorem represent fundamental steps; see \cite{K,Voloch1,Voloch2,Schoof,Szonyi,Giulietti}. The smallest possible size of complete arcs in projective planes $\mathrm{PG}(2,q)$ obtained using such algebraic methods is $O(q^{3/4})$; see \cite{Szonyi2}.

Almost nothing is known for the case $m>2$: sporadic  infinite families either  arise from the theory of $2$-character sets in $\mathrm{PG}(2, q)$  (see \cite[Sects. 12.2 and 12.3]{10} and \cite{8}) or  consist of the set of $\mathbb{F}_q$-rational points of some cubic curves \cite{7,14,20,Bartoli2016} (for $m=3$) or  curves of degree $m$  for specific $m$; see \cite{6}.  For $m = 3$ and $m=4$ smaller complete $m$-arcs of size $o(q)$ have been recently constructed in \cite{4,BSZ}; they consist of a suitable subset of $\mathbb{F}_q$-rational points of a curve of degree $m+1$.

In this paper we construct infinite families of complete $m$-arcs of size in the Hasse-Weil range, whose points are chosen, apart from few exceptions, among  the $\mathbb{F}_q$-rational points of curves of degree $m$.

The aim of our work is two-fold. On the one hand we show an effective method to provide examples of complete $m$-arcs from algebraic curves based on the investigation of the Galois closure of specific function field extensions.  On the other hand,  as an application of our machinery, we  show the existence of specific families of curves which provide examples of complete $m$-arcs in $\mathrm{PG}(2,q)$, with $q$ large enough with respect to $m$, of size in the interval $[q-2\lfloor m/2\rfloor \sqrt{q} +C_m,q+2\lfloor m/2\rfloor \sqrt{q} +C_m]$, where $C_m$ is a constant depending only on $m$. Families of $m$-arcs $\mathcal{A}$ whose  size satisfies  $0<\#\mathcal{A} -(q-2\lfloor m/2\rfloor \sqrt{q})=O(m)$ are provided in  Theorems \ref{Th:Asymptotic} and \ref{Th:minimal}. These complete $m$-arcs are the  smallest known in the literature so far. Of course, for certain particular values of $m$ and $q$ we expect that it will be possible to construct smaller $m$-arcs, but nevertheless our construction provides a general upper bound.

We also expect that the method can also be helpful in the construction of complete $m$-arcs of size $o(q)$ starting from  subsets of $\mathbb{F}_q$-rational points of curves of degree $m+1$, as already done for the cases $m=3$ and $m=4$ in \cite{4,BSZ}. 

More in details, our machinery can be summarized as follows. 

\begin{enumerate}
\item[(i)] Consider an irreducible curve $\mathcal{C}_f : f(x,y)=0$, where $f(x,y) \in \mathbb{F}_q[x,y]$ has degree $m$. Let $P\in \mathrm{AG}(2,q) \setminus \mathcal{C}_f(\mathbb{F}_q)$ be an affine point not belonging to $\mathcal{C}_f$. 
\item[(ii)] Let $\ell_{P,t} : y=t(x-a)+b$ be a line through $P$. Consider $F_{a,b}(t,x)=F_{P}(t,x)=f(x,t(x-a)+bx)$ and let $G^{geom}_{a,b}=\Gal(F_P(t,x):\overline{\mathbb{F}}_q(t))$.
\item[(iii)] Consider now the set of pairs $(a,b)\in \mathbb F_q^2$ such that $G^{geom}_{a,b}=\mathcal{S}_m$, the symmetric group over $m$ elements. Clearly, this forces $G^{arith}_{a,b}=\Gal(F_P(t,x):\mathbb{F}_q(t))=\mathcal S_n$ as $G^{geom}_{a,b}\leq G^{arith}_{a,b}$. If $q$ is large enough with respect to $m$, using Chebotarev density Theorem \cite{kosters2014short} one obtains the existence of  a specialization $t_0\in \mathbb{F}_q$ for which  $F_{a,b}(t_0,x)$ splits into $m$ pairwise distinct  linear factors $(x-x_i)$ defined over $\mathbb{F}_q$. This means that the intersection between $\ell_{P,t_0}$ and $\mathcal{C}_f$ consists of $m$ distinct $\mathbb{F}_q$ rational points $P_i=(x_i,t_0(x_i-a)+b)$ and therefore the point $P$ is $m$-covered by $\mathcal{C}_f(\mathbb{F}_q)$.
\item[(iv)] In order to deal with the set of all the points $(a,b)\in \mathrm{AG}(2,q)$ such that $G^{geom}_{a,b}\neq \mathcal{S}_m$,  we will have to add an additional set of points $\Gamma$ to $\mathcal C_f$. Practically, $\Gamma$ is contained in the union of a relatively small number of $\mathbb{F}_q$-rational lines and therefore we only need to add at most $O(m)$ points.
\end{enumerate}

One of the advantages of this method is that it allows us to  transfer geometric information to arithmetic one. In fact, we work in the algebraic closure of $\mathbb F_q$ to extract group theoretical properties of the geometric Galois group $G^{geom}$. By Galois correspondence, this group is contained in any arithmetic Galois group and  this will allow to extract properties of the pair  $(G^{geom},G^{arith})$.
In this case we simply want $G^{arith}=G^{geom}$ as we want the existence of a totally split place via Chebotarev Density Theorem.

In principle, any curve $\mathcal{C}_f$ of degree $m$ could be used to construct complete $m$-arcs, even if technical issues in {(iii)} can be hard to solve. In this paper, we first show how to produce complete $m$-arcs starting from the $\mathbb{F}_q$-rational points of the curve $y=x^m$ in order to better explain this method. All the complete $m$-arcs obtained in this way have size larger than $q$; see Corollary \ref{Cor:ArcToy}. 

The construction of examples of size smaller than $q$ needs the use of curves of genus larger than $0$. Therefore, in the second part of the paper, we consider curves of type $y^2=g(x)/x^{n}$, where $n<m-2$ and  $g(x)$ squarefree; see Theorems \ref{Th:main} and \ref{thm:chargreaterm}. Finally, in Section \ref{Section:5} we collect asymptotic results together with explicit examples of curves providing small complete $m$-arcs.

\section{Galois groups and Chebotarev Density Theorem}
In this section we briefly recall the Galois theoretical part of our approach which deals with totally split places. This method was successfully used also in \cite{micheli2019constructions} to build optimal locally recoverable codes of large dimension. Similar ideas were also applied to classify permutation rational functions in \cite{ferraguti2018full} and to prove factorization properties of polynomials over finite fields \cite{micheli2019selection}.

For a function field $M/\mathbb F_q$, we say that the field of constants of $M$ is the set of all elements of $M$ that are algebraic over $\mathbb F_q$.

As usual, when one wants to deal with spliting condition from a Galois theoretical perspective, one makes use of the following classical fact from algebraic number theory, whose proof can be found for example in \cite{guralnick2007exceptional}.

\begin{lemma}\label{orbits}
Let $L:K$ be a finite separable extension of function fields, let $M$ be its Galois
 closure and $G:= \Gal(M:K)$ be its Galois group. 
 Let $P$ be a place of $K$ and $\mathcal Q$ be the set of places of $L$ lying above $P$.
Let $R$ be a place of $M$ lying above $P$. Then we have the following:
\begin{enumerate}
\item There is a natural bijection between $\mathcal Q$ and the set of orbits of $H:=\mathrm{Hom}_K(L,M)$ under the action of the decomposition group $D(R|P)=\{g\in G\,|\, g(R)=R\}$.
\item  Let $Q\in \mathcal Q$ and let $H_Q$ be the orbit of $D(R|P)$ corresponding to $Q$. Then $|H_Q|=e(Q|P)f(Q|P)$ where $e(Q|P)$ and $f(Q|P)$ are ramification index and relative degree, respectively. 

\item The orbit $H_Q$ partitions further under the action of the inertia group $I(R|P)$ into $f(Q|P)$ orbits of size $e(Q|P)$. 

\end{enumerate}
\end{lemma}

The following can also be deduced by \cite{kosters2014short}; for completeness, we include a proof here for the case we need.
\begin{theorem}\label{thm:existence_tot_split}
Let $p$ be a prime number, $m$ a positive integer, and $q=p^m$.
Let $L:F$ be a separable extension of global function fields over $\mathbb F_q$ of degree $n$,  $M$ be the Galois closure of $L:F$, and suppose that the field of constants of $M$ is $\mathbb F_q$.
There exists an explicit constant $C\in \mathbb R^+$ depending only on the genus of $M$ and the degree of $L:F$ such that if $q>C$ then $L:F$ has a totally split place.
\end{theorem}
\begin{proof}
First, observe that $L:F$ has a totally split place if and only if its Galois closure $M:F$ has a totally split place \cite[Lemma 3.9.5]{17}.
So we are left with counting the degree $1$ totally split places of $F$ in $M:F$ and ensure that there is at least one. By looking at Lemma \ref{orbits}, a place $P\subseteq F$ is totally split in $M:F$ if and only if $D(R|P)=\{\mathrm{Id}\}$ for any place $R$ of $M$ lying above $P$ (all the decomposition groups are conjugate when $R$ varies over $P$).
This means that we have to count the unramified places of $F$ of degree $1$ such that the identity is a Frobenius. Suppose that the  set of totally split places of $F$ is $T$, then
 every element of $T$ lies below a place of degree one of $M$. Moreover, every place of degree $1$ of $M$ lies above a place of degree $1$ of $F$ simply by multiplicativity of relative degree. Also, above every unramified place of degree $1$, say $P$, of $T$, exactly $|G|$ unramified places of degree $1$ of $M$ lie above $P$ and viceversa, because the extension is Galois.
 
Now, by Hasse-Weil Theorem the places $\mathcal P^1(M)$ of degree $1$ of $M$  are at least
$q+1-2g_M\sqrt q$. By the consideration above we also have that $|\mathcal P^1(M)|=|T|\cdot |G|+ |R|$ where $R$ are the ramified places, so the number of totally split places satisfies
$|T|\geq (q+1-2g_M\sqrt q-|R|)/|G|$.

Let us now give an upper bound for $|R|$. First, observe that a place of $F$ ramifies in $L:F$ if and only if it ramifies in $M:F$ by standard Galois theory. 
Second, observe that $|R|$ can be upper bounded with the number of ramified  places of $L:F$ over $\overline {\mathbb F}_q$ (so that the degree of every place becomes $1$). Using \cite[Corollary 3.5.6]{17} we obtain that
\[2g_L\geq 2g_L-2\geq n(2g_F-2)+\sum_{P\subseteq F} \underbrace{\sum_{P'\mid P} (e(P'|P)-1)}_{\geq 1}\geq -2n+|R|.\]
Therefore, we get $|R|\leq 2(g_L+n)$.

Finally, the geometric  genus $g_M$ is independent of the size of the base field:  there exists a constant $C=C(g_M,n)$ such that if $q>C$ then 
\begin{equation}\label{eq:totsplitcond}
(q+1-2g_M\sqrt q-2(g_L+n))/|G|\geq 1,
\end{equation}
ensuring the existence of a totally split place.
The constant can be made explicit by upper bounding $|G|$ with $n!$, and  $g_M$ using recursively Castelnuovo inequality \cite[Theorem 3.11.3]{5} (in fact, $M$ can be seen as the compositum of the images of $L$ via the elements of $G$).
\end{proof}

\begin{corollary}\label{cor:constant}
If $p>n$, the constant $C$ appearing in Theorem \ref{thm:existence_tot_split}  can be chosen to be equal to $9(g_F+g_L+n)^2(n!)^2$.
\end{corollary}
\begin{proof}
We simply have to make the estimate on $g_M$ in the proof of Theorem \ref{thm:existence_tot_split}  explicit using Hurwitz genus formula in the tame case \cite[Corollary 3.5.6]{17}, which is allowed because $p>n$  (and therefore none of the ramification indexes can be divided by $p$).
Whenever $p> n$, by Hurwitz genus formula and Hilbert's Different Exponent Theorem we have that 
\[2g_M-2= [M:F](2 g_F-2)+\sum_P \sum_{P'\mid P} (e(P'\mid P)-1).\]
Since now $[M:F]=|G|\leq n!$, we have that
\[2g_M\leq n!2 g_F+n! |R|\leq 2(g_F + g_L+n)(n!),\]
so that $g_M\leq (g_F + g_L+n)(n!)$.

For example, we can take $C\geq 9(g_F+g_L+n)^2(n!)^2$ so that \eqref{eq:totsplitcond} is verified.

\end{proof}

\begin{remark}
It is surely possible to refine the constant in Corollary \ref{cor:constant} at the price of having more complicated expression. We prefer to avoid it for the sake of readability.
\end{remark}

Another result we will use is Theorem  \ref{thm:hemultsn} below, which is purely group theoretical. We report it here for the reader's convenience:
\begin{theorem}\label{thm:hemultsn}\cite[Theorem 3.9]{Helmut}
Let $r$ be a prime and $G$ be a primitive group
of degree $n = r + k$ with $k \geq 3$. If $G$ contains an element of
degree and order $r$ (i.e. an $r$-cycle), then $G$ is either alternating or symmetric.
\end{theorem}

\section{Complete $m$-arcs from rational curves $y=x^m$} 
First, we show how use the machinery {(i)}-{(iv)} in the easiest possible case, that is when $\mathcal{C}_f$ is the rational curve of degree $m$ of equation $y-x^m=0$. For a fixed point $P=(a,b)\in AG(2,q)$ let $\ell_{a,b,t} : y=t x-t a+b$, be a line through $P$ with slope $t$. Let $\overline{\mathbb{F}_q}$ denote the algebraic closure of $\mathbb{F}_q$. 

First we prove a useful geometrical feature for the curve $\mathcal{C}_f$.

The following proposition ensures that the specialization of $x^m-t x+t a-b$ at a certain $\overline t\in \overline{\mathbb F_q}$ has exactly one factor of multiplicity  $2$, and all the other factors with multiplicity $1$.

\begin{proposition}\label{Th:DoublePointToy}
Let $m\geq 5$, $p\nmid m(m-1)$. For any $P=(a,b)\in AG(2,q)$, $b\neq 0$, there exists $t \in \overline{\mathbb{F}_q}$ such that $|\mathcal{C}_{f} \cap\ell_{a,b,t}|=m-1$. 
\end{proposition}
\proof
We have to prove the existence of $\overline t \in \overline{\mathbb{F}_q}$ for which $x^m-\overline t x+\overline t a-b$ has exactly $m-1$ distinct roots, which is equivalent to show that the polynomial $x^m-\overline t x+\overline t a-b$ has exactly one root with multiplicity $2$ and all the others with multiplicity $1$. 

Let $F_{a,b,t}(x)=x^{m}-t x+t a-b$. First, note that $F_{a,b,t}(x)$, $t \neq 0$, has roots  of multiplicity at most two. In fact, $(\partial F_{a,b,t}/\partial x)(x)=mx^{m-1}-t$ is squarefree ($p\nmid (m-1)$ by assumption) and therefore  $F_{a,b,t}(x)$ has no roots of multiplicity three. Take now a solution $(\overline x, \overline t)$ of the system $F_{a,b,\overline t}(x)=0$ and $(\partial F_{a,b,t}/\partial x)(x)=0$, so that $F_{a,b,\overline{t}}(x)$ has $\overline{x}$ as a double root. Clearly, we must have that $\overline{t}=m\overline{x}^{m-1}$ and therefore $\overline x$ is a root of $(1-m)x^{m}+am x^{m-1}-b$.

If $\overline{t}\notin \{0,m\}$, we prove now that $F_{a,b,\overline t}$ cannot have two roots $x_1\neq x_2$ of multiplicity two. In fact,
$$\left(\frac{\partial F_{a,b,\overline{t}}}{\partial x}\right)(x_1)=mx_1^{m-1}-\overline{t}=0=mx_2^{m-1}-\overline{t}=\left(\frac{\partial F_{a,b,\overline{t}}}{\partial x}\right)(x_2)$$
and 
 \begin{eqnarray*}
    F_{a,b,\overline{t}}(x_1)&=&x_1^m-\overline{t} x_1+\overline{t} a-b=
    (\overline{t}/m-\overline{t})x_1+\overline{t}a-b=0\\
    &=&(\overline{t}/m-\overline{t})x_2+\overline{t}a-b=x_2^m-\overline{t} x_2+\overline{t} a-b=F_{a,b,\overline{t}}(x_2),
    \end{eqnarray*}
   which yields $(\overline{t}/m-\overline{t})x_1=(\overline{t}/m-\overline{t})x_2$ and thus $x_1=x_2$, a contradiction since we assumed $\overline{t}\notin \{0,m\}$. 
 
Now we simply have to show that we can always choose $\overline{t}\notin \{0,m\}$. First, we observe that $\varphi_{a,b}=(1-m)x^m+amx^{m-1}-b$  has  $m$ distinct roots: $0$ is not a root and the unique possible repeated root is $a$, which is not a root since $b\neq a^m$ ($P$ does not belong to $\mathcal{C}_f$).  In addition, we observe that at most one  root $\overline{x}$ of $\varphi_{a,b}$ satisfies $\overline{x}^{m-1}=1$.  So that there is necessarily another root $\overline{x}$ of $\varphi_{a,b}$ for which $\overline{t}=m\overline{x}^{m-1}\notin\{ 0,m\}$. 

The claim follows.
\endproof

For any polynomial $F\in \mathbb F_q(t)[x]$ we can define the \emph{arithmetic} Galois group as
$\Gal(F\mid \mathbb F_q(t))$ and the \emph{geometric} Galois group as $\Gal(F\mid \overline{\mathbb F}_q(t))$. It is immediate to see that $\Gal(F\mid \overline{\mathbb F}_q(t))\cong \Gal(F\mid k(t))$, where $k$ is the field of constants of the splitting field $M$ of $F$ over $\mathbb F_q(t)$.

\begin{theorem}\label{Th:PrimitiveToy}
The Galois group of $F=x^m-b-t(x-a)$ over $\mathbb F_q(t)$ is primitive.
\end{theorem}
\begin{proof}
Let $M$ be the splitting field of $F$ and $G$ be the Galois group of $F$ over $\mathbb{F}_q(t)$. Let $x$ be a root of $F$ and consider the extension $\mathbb{F}_q(x)/\mathbb F_q(t)$. Clearly, $t=(x^m-b)/(x-a)=f$ by definition. Let us assume by contradiction that $G$ is not primitive, then it is a standard fact that the stabilizer $S_x\leq G$ of $x$ is not maximal \cite[Theorem 8.2]{Helmut}. Let us consider the fixed field $L$ of $S_x$. By the Galois correspondence, we have the chain of extension $\mathbb{F}_q(x)/L/\mathbb{F}_q(f)$ with both $[\mathbb{F}_q(x):L],[L:\mathbb{F}_q(f)]$ greater than or equal to $2$. We know that every subfield of $\mathbb{F}_q(x)$ is rational by Luroth theorem. This forces $L=\mathbb{F}_q(s)$ for some $s=h/k$, with coprime polynomials $h,k\in \mathbb{F}_q[x]$. But now, also the field $\mathbb{F}_q(s)$ is rational and $\mathbb{F}_q(f)$ is a subfield, which forces $f=j(s)/g(s)$ for some coprime $j,g\in \mathbb F_q[Y]$. Since $[\mathbb{F}_q(x):L],[L:\mathbb{F}_q(f)]\geq 2$ it is immediate to observe that 
\begin{equation}\label{Eq:assumptions}
2\leq  \deg(h/k),\deg(j/g)\leq m-1.
\end{equation}

We are now ready to deduce the contradiction.

First, we  show that $g=(x-u)^d$ for some positive integer $d$ and $u\in \overline{ \mathbb {F}}_q$. Suppose that $g$ has at least two coprime factors,  that is $g=(x-z_0)^{i_1}(x-z_1)^{i_2}p(x)$.
We distinguish two cases:
\begin{itemize}
    \item $z_0$ (or $z_1$) is such that $h(x)-z_0k(x)=\overline{z}\in \overline{\mathbb{F}}_q$ (this means that $h/k-z_0$ has no roots in $\overline{\mathbb{F}}_q$). In this case   $h(x)-z_1k(x)$ (resp. $h(x)-z_0k(x)$) is a polynomial of degree $2\leq \deg(h)=\deg(k)\leq m-1$ and so it has either (at least) two distinct roots in $\overline{\mathbb{F}}_q$ or a repeated root of multiplicity larger than one. 
    \item Both $h(x)-z_0k(x)$ and $h(x)-z_1k(x)$ are not constant. By the argument above there are at least two distinct roots in $\overline{\mathbb{F}}_q$ of $(h(x)-z_0k(x))(h(x)-z_1k(x))$.
\end{itemize}

This shows that $g(h/k)=(h/k-z_0)^{i_1}(h/k-z_1)^{i_2}p(h/k)$ has either one root in $\overline{\mathbb{F}}_q$ of multiplicity larger than one or two distinct roots in $\overline{\mathbb{F}}_q$.

Now observe that $j(z_0)\neq 0$ and $j(z_1)\neq 0$ by the coprimality conditions of $j$ and $g$, which leads to either at least two poles  at finite of $f$ or a pole at finite with multiplicity at least two of $f$, which is impossible because $f=(x^m-b)/(x-a)$.

So we have that $g=(x-z_0)^d$ and therefore situation
\[\frac{j(h/k)}{(h/k-z_0)^d }=\frac{x^m-b}{x-a}.\]

Note that $h(x)/k(x)-z_0$ cannot have roots $x_0\in \overline{\mathbb{F}}_q\setminus \{a\}$: in this case $x_0$ would be a pole $f$ distinct from $a$ and $\infty$, a contradiction.

So the roots of $h(x)/k(x)-z_0$ are either $a$ or $\infty$. If $\infty$ is the unique root, then it must have order $\deg(h/k)$. In this case we conclude by observing that we have $d\deg(h/k)=m-1$ (as $\infty$ is a pole of order $m-1$ of $f$ and $j(z_0)\neq 0$). But then by the tower law we also have that $\deg(h/k)|m$, which forces $\deg(h/k)=1$, a contradiction to \eqref{Eq:assumptions}.

Therefore, we necessarily have $(h/k)(a)=z_0$ and $(h/k)(a^{\prime})\neq z_0$ for any $a^{\prime}\neq a$. In this case $a$ is a pole of order at least $d$ (remember $j((h/k)(a))=j(z_0)\neq 0$ as $j,g$ are coprime), which forces $d=1$. 
Moreover, $a$ must be a zero of order one of $h/k-z_0$ given the form of $f$.

Since now $\deg(s)=\deg(h/k)\geq 2$, the equation $(h/k)(x)=z_0$ must have another solution, which has to be $\infty$.
Moreover, since $\infty$ is a pole of order $m-1$ for $f$ and $j(z_0)\neq 0,\infty$ (in fact, $j$ is a polynomial coprime with $g$),  $\infty$  must have order $m-1$ as a zero of $(h/k)(x)-z_0$. So that $h/k$ is a map of degree $m$ (since it has zero divisor of degree $m$), again a contradiction to \eqref{Eq:assumptions}.
\end{proof}

\begin{theorem}\label{Th:MainToy}
Let $h$ be a positive integer, $p$ a prime number and $q=p^h$. 
Suppose that $p \nmid m(m-1)$. Let $a,b\in \mathbb F_q$ and $F_{a,b}(x)=x^m-tx+ta-b\in \mathbb F_q(t)$. For any $b\neq 0$, $\Gal(F_{a,b}(x)\mid \mathbb F_q(t))=\Gal(F_{a,b}(x)\mid \overline{\mathbb F}_q(t))=\mathcal{S}_m$.
\end{theorem}
\begin{proof}

Using Proposition \ref{Th:DoublePointToy} 
we obtain that there is a positive integer $d$ and a place $t_0\in \mathbb F_{q^d}$
such that $x^m-t_0x+t_0a-b$ has exactly one factor of multiplicity 
$2$ and all the others of multiplicity $1$. Let now $M$ be the splitting field of 
$F_{a,b}(x)$ over $\mathbb F_{q^d}(t)$. Let $R$ be a place of $M$ lying above $t_0$.
Now, using Lemma \ref{orbits} we obtain that the decomposition group $D(R\mid t_0)$ has a cycle of order exactly $2$ and fixes all the other elements of $H=\mathrm{Hom}_{\mathbb F_q(t)}(\mathbb F_q(t)[x]/F_{a,b}(x),M)$ ($H$ can be simply thought as the set of roots of $F_{a,b}(x)$ in $\overline{\mathbb F_q(t)})$. Now pick any element $g\in D(R\mid t_0)$ that acts non-trivially on $H$. This element has to be a transposition, which in turn forces $\Gal(F_{a,b}(x):\mathbb F_{q^{du}}(t))$ to contain a transposition for any $u\in \mathbb N$ and therefore in particular that $\Gal(F_{a,b}(x)\mid \overline{\mathbb F}_q(t))$ contains a transposition.

Using now Theorem \ref{Th:PrimitiveToy} we get that $\Gal(F_{a,b}(x)\mid \overline{\mathbb F}_q(t))$ is also primitive.

Finally, using Theorem  \ref{thm:hemultsn} with $r=2$ we conclude that both $\mathcal{S}_m=\Gal(F_{a,b}(x)\mid \overline{\mathbb F}_q(t))$ and $\Gal(F_{a,b}(x)\mid \mathbb F_q(t))=\mathcal S_m$.
\end{proof}

Let $q=p^e$, for some integer $e\in \mathbb N$, and let $C$ be the constant obtained in Theorem \ref{thm:existence_tot_split} with  $F=\mathbb F_q(t)$ and $L=F(z)$, where $z$ is a root of $F_{a,b}(x)\mid \overline{\mathbb F}_q(t)$. If $p>m$ then $C$ can be chosen according to Corollary \ref{cor:constant}, obtaining $C=9m^2(m!)^2$.
\begin{corollary}\label{Cor:ArcToy}
Suppose that $p\nmid m(m-1)$ and $q\geq C$. There exists a complete $m$-arc of size $q+2m-1$ in $PG(2,q)$.
\end{corollary}
\begin{proof}
Set $F=\mathbb F_q(t)$ and $L=F(z)$.
By Theorem \ref{Th:MainToy} we have that the field of constants of the Galois closure of $L:F$ is trivial, as the geometric Galois group of $F_{a,b}(x)$ is equal to the arithmetic one. Using now Theorem \ref{thm:existence_tot_split} we have that there exists a specialization $t_0\in  \mathbb F_q$ such that $F_{a,b}(x)$ is totally split. This is equivalent to the fact that through any point $P=(a,b)\in \mathrm{PG}(2,q)$, $b\neq 0$, not belonging to $\mathcal{C}_f$ there exists an $\mathbb{F}_q$-rational line intersecting  $\mathcal{C}_f$ in $m$ distinct $\mathbb{F}_q$-rational points. This means that $\mathcal{C}_f(\mathbb{F}_q)$ $m$-covers all the affine points outside $y=0$. Therefore, adding to $\mathcal{C}_f(\mathbb{F}_q)$ at most $m-1$ points on $y=0$ and $m-1$ points at infinity, we can obtain  a complete $m$-arc of size at most $q+2m-1$.
\end{proof}

\begin{remark}
The reader should notice here how geometric properties (so over $\overline{F}_q$), such as the ones that involve ramification of a covering over $\overline{\mathbb{F}_q}$, given by Proposition \ref{Th:DoublePointToy}, transfer via  Proposition \ref{thm:existence_tot_split} and Lemma \ref{orbits} into a totally arithmetic statement, involving the existence of a totally split place over $\mathbb F_q$.
\end{remark}

\section{Construction of complete $m$-arcs from curves $y^2x^{n}=g(x)$}\label{Section:y2}
In this section we always consider $p$ an odd prime. Complete $m$-arcs obtained in the previous section have size larger than $q$. Let us denote by $A(K,m)$  the smallest size of a complete $m$-arc of $\mathrm{PG}(2,K)$. In order to construct examples of complete $m$-arcs of size smaller than $q$ we have to use curves of genus larger than zero. The easiest choice is given by curves of the type $y^2=\widetilde{g}(x)$, with $\widetilde{g}(x)\in \mathbb{F}_q(x)$. In particular, we focus on curves $\mathcal{C}_f : f(x,y)=0$, where $f(x,y)=y^2x^n-g(x)$, for some $n<m-2$ with $n+ m\equiv 0 \pmod 2$, and a  squarefree polynomial $g(x)$ of degree $m$. In this case, the genus $g(\mathcal{C}_f)$ equals $\lfloor m/2\rfloor$; see \cite[Corollary 3.7.4]{17}.

In what follows we consider a prime $r\in ]m/2,m-2]$ such that $p\nmid r(r+2)(m-r-2)$. Let us discuss the existence of such prime $r$. 
\begin{remark}\label{Remark:Bertrand}
First, observe that if $p>m$ the existence of such prime is obvious thanks to  Bertrand's postulate for all $m\geq 8$. Using a generalization of Bertrand's postulate due to Nagura \cite{Nagura}, we can assume that there exists a prime number between $n$ and $6n/5$ for any $n\geq 25$.
 This yields the existence of at least three primes  in the interval $[m/2,(1.2)^3m/2]\subset [m/2,m-2]$ for $m\geq 50$. Actually, a computer search shows that for any  $m\in \{15\} \cup \{19\ldots 49\}$ there exist three primes in the interval $[m/2,m-2]$. We show now that $p>m/2$ is a sufficient condition for all $m\in \mathbb N$ such that $m\geq 50$ or $m\in \{15\} \cup \{19\ldots 49\}$. To see this, observe that there are $2$ primes $p_1$ and $p_2$ in the interval $[m/2,m-2]$ that are distinct from $p$. Moreover, $m-r-2$ cannot be divisible by $p$ for any of such primes, as it is less than $m/2$. Finally, if $p_1,p_2$ both satisfy $p\mid p_i+2$ then they are both solutions of $x\equiv -2 \mod p$, but then they have to be congruent modulo $p$ which is a contradiction as $p>m/2$. A closer  analysis of the  remaining $m$'s  shows that $p\geq m/2$ would provide a sufficient condition for any $m\geq5$.
\end{remark} 

Let us denote by $\mathcal{C}_f$ the curve $x^{m-2-r}y^2=g(x)$, $g(0)\neq 0$, and $g(x)$ squarefree.

Let $P=(a,b) \in \mathrm{AG}(2,q) \setminus \mathcal{C}_f$.
Consider a line $y=t(x-a)+b$ through $P$ and denote by 
\begin{equation}\label{Eq:g}
H_{a,b}(t,x)=t^2(x-a)^2x^{m-2-r}+2bt(x-a)x^{m-2-r}+b^2x^{m-2-r}-g(x).
\end{equation}

We want to obtain a result similar to Theorem \ref{Th:MainToy} for the Galois group of $H_{a,b}(t,x)$.

\begin{proposition}\label{Prop:1}

For any $(a,b)\in (\mathbb{F}_q)^2$, the Galois  group  of $H_{a,b}(t,x)$ is either 
$\mathcal{S}_m$ or $\mathcal{A}_m$.

\end{proposition}

\begin{proof}
Let $P_{\infty}$ be the pole of $t$ in $\mathbb{F}_q(t)$. For any place $Q_{\infty}$ lying over $P_{\infty}$ in $\mathbb{F}_q(t,x):\mathbb{F}_q(t)$, where $H_{a,b}(t,x)=0$, we have that
$$v_{Q_{\infty}}(t)=e(Q_{\infty}|P_{\infty})v_{P_{\infty}}(t)=-e(Q_{\infty}|P_{\infty}).$$ For at least one place $\overline{P}$ lying over $P_{\infty}$ one have $v_{\overline{P}}(x)=s<0$. We want to understand the ramification of the extension of places $P_\infty \subseteq \overline P$.
Using the equation $H_{a,b}(t,x)=0$ and the fact that
$$-2e(\overline P|P_{\infty})+(m-r)s<-e(\overline P|P_{\infty})+(m-r-1)s<(m-r-2)s$$ we obtain that
\begin{eqnarray*}
\min\{-2e(\overline P|P_{\infty})+(m-r)s,-e(\overline P|P_{\infty})+(m-r-1)s,(m-r-2)s\}&=&\\
v_{\overline{P}}\left(t^2(x-a)^2x^{m-2-r}+2bt(x-a)x^{m-2-r}+b^2x^{m-2-r}\right)&=&\\
v_{\overline{P}}\left(g(x)\right)&=&ms.
\end{eqnarray*}

Now,
$$-2e(\overline P|P_{\infty})+(m-r)s=ms,$$
which yields $-2e(\overline P|P_{\infty})=rs$. Since $r$ is an odd prime,  $s$ must be even. Also, $2e(\overline P|P_{\infty}) \leq 2m <4r$, so $|s|<4$ and therefore $s=-2$ and $e(\overline P|P_{\infty})=r$.

Let $M$ be the splitting field of $H_{a,b}(x,t)$ over $K=\overline{\mathbb F}_q(t)$, let $\{z_1,\dots,z_m\}\subset M$ be the set of roots of $H_{a,b}(x,t)$, and let $L=K(z_1)$.  Notice that in this set-up all relative degrees of extension of places are $1$ because we are working over the algebraic closure of $\mathbb F_q$.
Using Lemma \ref{orbits} for the extension $L:K$, we get that $D(R|P)\leq \Gal(M:K)$ has an orbit $\mathcal{O}$ on $\{z_1,\ldots,z_m\}$ of size $r>m/2$ (we are using here the standard identification of the set of roots of $H_{a,b}(x,t)$ with $\mathrm{Hom}_K(L,M)$, which is of course compatible with the $D(R|P)$ action). Consider an element $\phi\in D(R|P)$ acting non-trivially on $\mathcal{O}$. This means  that it acts on $\mathcal{O}$ as cycle of length $r$. Then $\phi^{(r-1)!}$ is an element of $G$ which acts on $\mathcal{O}$ as cycle of length $r$ and trivially outside. Since $r>m/2$ this also shows that $G$ is primitive. By Theorem \ref{thm:hemultsn}, $G$ is either $\mathcal{S}_m$ or $\mathcal{A}_m$.
\end{proof}

In order to show that $G$ is not $\mathcal{A}_n$, we prove that actually $G$ contains a transposition. In particular, we show the existence of  a specialization $t_0$ for which $H_{a,b}(t_0,x)$ has only simple roots apart from a  unique double one. 

To this end, we start with a technical lemma.
\begin{lemma}\label{Lemma}
Let $p$ be an odd prime number, $m$ be a positive integer, $r$ be another prime such that $m/2<r\leq m-2$, and suppose that $p\nmid r(r+2)(m-r-2)$. Let $g(x)$ be a  squarefree polynomial with leading coefficient $a_0$ such that $g(0)\neq 0$.  Consider  
$$h(x)=-(m-r-2)g(x)+xg^{\prime}(x)=(r+2)a_0x^m+\cdots.$$
Then the two curves $\mathcal{D}_1$ and $\mathcal{D}_2$ defined respectively by 
\begin{eqnarray*}(h(y))^2x^{m-r}g(x)&=&(h(x))^2y^{m-r}g(y),\\
2yh(x)g(y)&=&h(y)((y-x)h(x)+2xg(x)).
\end{eqnarray*}
do not share a component different from $x-y=0$.
\end{lemma}

\begin{proof}
By our assumptions, $h(x)$ is of degree $m$, $h(0)\neq 0$, and $\gcd(g(x),h(x))=1$ since $g(x)$ is squarefree and $g(0)\neq 0$.

Note that  $\ell: x-y=0$ is  a component of both $\mathcal{D}_1$ and $\mathcal{D}_2$. Since $(1:1:0)$ is a simple point of $\mathcal{D}_2$,  points at infinity of $\mathcal{D}_2\setminus \ell$ are $P_{\infty}=(1:0:0)$ and $Q_{\infty}=(0:1:0)$ and they both belong to $\mathcal{D}_1$ too.  

Let $\mathcal{D}$ be a common component of $\mathcal{D}_1$ and $\mathcal{D}_2$ distinct from $\ell$. We use the following argument: we take a line $s$ and intersect it with both $\mathcal D_1$ and $\mathcal D_2$. If $\mathcal D_1$ and $\mathcal D_2$ share $\mathcal D$ as a common component, we must have that the tangent cone at  least one point $P\in s\cap \mathcal D_1\cap \mathcal D_2$ shares a line with the tangent cone of $\mathcal D_i$ for both $i=1,2$.
If we can show that the tangent cone of $\mathcal D_1$ at $P$ never shares a line with the tangent cone of $\mathcal D_2$ at $P$, for any $P \in s\cap \mathcal D_1\cap \mathcal D_2$,  we are done. In the argument below we consider as $s$ the line at infinity $\ell_{\infty}$.

The component $\mathcal D$ can intersect  $\ell_{\infty}$ at $P_{\infty}=(1:0:0)$ and $Q_{\infty}=(0:1:0)$.

We want to determine the tangent cone at $Q_{\infty}$ of both $\mathcal{D}_1$ and $\mathcal{D}_2$. 
Let $ \overline{h}(x,T)$ and $ \overline{g}(x,T)$ be the homogenizations of $h(x)$ and $g(x)$. Consider the projectivity $(x:y:T) \mapsto (x:T:y)$ which interchanges  $Q_{\infty}$ and $(0:0:1)$ and apply it to both $\mathcal{D}_1$ and $\mathcal{D}_2$. We obtain 
\begin{eqnarray*}
( \overline{h}(T,y))^2x^{m-r}\overline{g}(x,y)=( \overline{h}(x,y))^2T^{m-r}\overline{g}(T,y),\\
2T \overline{h}(x,y)\overline{g}(T,y)= \overline{h}(T,y)((T-x) \overline{h}(x,y)+2x\overline{g}(x,y)).
\end{eqnarray*}
Since $\overline{g}(1,0)=a_0$, $ \overline{h}(1,0)=(r+2)a_0$, the tangent cones at the origin  $(0:0:1)$ are 
\begin{eqnarray*}
 \overline{h}(1,0)^2x^{m-r}\overline{g}(x,y)-(r+2)^2a_0^2x^{m-r}\overline{g}(x,y)\not\equiv 0,\\
(2\overline g(1,0)- \overline{h}(1,0)) \overline{h}(x,y)+ra_0 \overline{h}(x,y)\not\equiv 0.
\end{eqnarray*}
The two tangent cones do not share any factor, since $p\nmid (m-r-2)$ and $g(0)\neq 0$ and thus $x\nmid \overline{h}(x,y)$, and $\gcd(g(x),h(x))=1$. This means that $\mathcal{D}$ does not pass  through $Q_{\infty}$. 

Since $(x:y:T)\mapsto (y,x,T)$ is an automorphism of both $\mathcal D_1$ and $\mathcal D_2$, we have that $\mathcal{D}$ cannot contain the point $P_{\infty}$ too. But by Bezout's theorem we have that $\emptyset \neq \left(\ell_{\infty} \cap \mathcal{D}\right)$. This is a contradiction, since
$$\emptyset \neq \left(\ell_{\infty} \cap \mathcal{D}\right) \subseteq  \left(\ell_{\infty} \cap \mathcal{D}_1  \cap \mathcal{D}_2\right)  =\emptyset.$$
\end{proof}

A key point in our argument is to prove the existence through points $(a,b)\notin \mathcal{C}_f$ of tangent lines to $\mathcal{C}_f$ intersecting the curve in $m-1$ distinct points. The following proposition provides an upper bound to the number of tangent lines to $\mathcal{C}_f$ intersecting $\mathcal{C}_f$ at two  distinct points of tangency or at a point with multiplicity larger than $2$. 

\begin{proposition}\label{Prop:TangentLines}
Let $p$ be an odd prime number, $q$ be a power of $p$, $m$ be a positive integer, $r$ be another prime such that $m/2<r\leq m-2$, and suppose that $p\nmid r(r+2)(m-r-2)$. Consider   a squarefree polynomial $g(x)\in \mathbb F_q[x]$ of degree $m$ such that $g(0)\neq 0$.
At most $7m^2+3m+2$ tangent lines to the curve $\mathcal{C}_f: x^{m-r-2}y^2=g(x)$ intersect   the curve in less than $m-1$ distinct points.  
\end{proposition}
\begin{remark}
The proof of this proposition is a bit technical, we first want to outline the idea and the structure. First, we consider the tangent lines which pass through a point of the curve of the form $(x_0,0)$: these are tangents with  multiplicity higher than $2$, so they do intersect the curve at less than $m-1$ distinct points and therefore we have to exclude them. Then, we consider the tangent lines that are passing through an affine point $(x_0,y_0)$ with $y_0\neq 0$ and show that the ones that are tangent to $C_f$ also in another point are only a few, which we count as well. Finally the remaining lines that intersect the curve in less than $m-1$ distinct points are the ones that are tangent to a point $(x_0,y_0)$ with multiplicity larger than or equal to $3$, which are also only a few.
Summing up the number of all these kinds of lines  we conclude (any other tangent line to $\mathcal C_f$ is going to be tangent only at one point and with multiplicity exactly $2$).
\end{remark}
\begin{proof}
First, note that there are no affine singular points in $\mathcal{C}$. There are no points of type $(0,y_0)$ and the tangent lines at $(x_0,0)$, where $g(x_0)=0$, are the lines $x=x_0$, which intersect $\mathcal{C}_f$ only at $(x_0,0)$ and $(0:1:0)$. 
Let us then consider the tangent line $\ell_P$  at a point $P=(x_0,y_0)\in \mathcal{C}$ such that $x_0y_0\neq 0$.
The parametric equation for the tangent line $\ell_P$ is 
\begin{equation}\label{Eq:TangentLine}
\left\{
\begin{array}{l}
x=x_0+2x_0^{m-r-1}y_0t \\
y=y_0+\Big(-(m-r-2)g(x_0)+x_0g^{\prime}(x_0)\Big) t\\
\end{array}
\right..
\end{equation}

Denote by $h(x)$ the polynomial $-(m-r-2)g(x)+xg^{\prime}(x)=(r+2)a_0x^m+\cdots$. By our assumption on $r+2$ and $g(x)$ we see that $\deg(h(x))=m$.

Consider now another tangent line $\ell_Q$, where $Q=(x_1,y_1)\in \mathcal{C}$. Then $\ell_Q \equiv \ell_P$ if and only 

\begin{eqnarray}
\frac{x_1-x_0}{2x_0^{m-r-1}y_0}=\frac{y_1-y_0}{h(x_0)} \label{Eq:1},\\
\frac{x_0^{m-r-1}y_0}{h(x_0)} =\frac{x_1^{m-r-1}y_1}{h(x_1)}\label{Eq:2}.
\end{eqnarray}

Equation \eqref{Eq:1} yields 
\begin{eqnarray}y_1&=&\frac{(x_1-x_0)h(x_0)}{2x_0^{m-r-1}y_0}+y_0=\frac{(x_1-x_0)h(x_0)}{2x_0g(x_0)}y_0+y_0\nonumber\\
&=&\frac{(x_1-x_0)h(x_0)+2x_0g(x_0)}{2x_0g(x_0)}y_0\label{Eq:3}.
\end{eqnarray}

From \eqref{Eq:2} we get 

\begin{equation}\label{Eq:4}\frac{y_1}{y_0}=\frac{h(x_1)x_0^{m-r-1}}{h(x_0)x_1^{m-r-1}}\end{equation}

and so, by squaring and using the equation of $\mathcal C_f$, we obtain that
\begin{equation}\label{Eq:5}(h(x_1))^2x_0^{m-r}g(x_0)=(h(x_0))^2x_1^{m-r}g(x_1).\end{equation}

Combining \eqref{Eq:3} and \eqref{Eq:4} we get 
\begin{equation}\label{Eq:6}2h(x_1)x_0^{m-r}g(x_0)=h(x_0)x_1^{m-r-1}((x_1-x_0)h(x_0)+2x_0g(x_0)) .\end{equation}

Thanks to \eqref{Eq:5}, $h(x_1)x_0^{m-r}g(x_0)=(h(x_0))^2x_1^{m-r}g(x_1)/h(x_1)$, and from \eqref{Eq:6} one gets 
\begin{equation}\label{Eq:7}2h(x_0)x_1g(x_1)=h(x_1)((x_1-x_0)h(x_0)+2x_0g(x_0)).
\end{equation}

Consider now the curves $\mathcal{D}_1$ and $\mathcal{D}_2$ defined by \eqref{Eq:5} and \eqref{Eq:7}. By Lemma \ref{Lemma}, they do not share a component different from $x_1-x_0=0$. By Bezout's Theorem, they intersect in at most $2m(4m-r-1)$ points with $x_1\neq x_0$ (counted with multiplicity). 

Since all affine points of $\mathcal{C}_f$ are simple, there passes exactly one tangent through each of them. Therefore there are at most $2m(4m-r-1)$ tangent lines intersecting $\mathcal{C}_f$ in at least two points of tangency. 

Now we deal with tangent lines $\ell_P$ at points $P=(x_0,y_0)\in \mathcal{C}_f$ intersecting $\mathcal{C}_f$ at $P$ with multiplicity larger than $2$. Recall that we are assuming $x_0y_0\neq 0$. 

Easy computations from the parametric equations \eqref{Eq:TangentLine}  show
\begin{eqnarray*}
g(x) &=& g(x_0)+2x_0^{m-r-1}y_0g^{\prime}(x_0)t+2x_0^{2m-2r-2}y_0^2 g^{\prime\prime}(x_0)t^2+\cdots,\\
y^2&=&y_0^2+2y_0h(x_0)t+h(x_0)^2t^2,\\
x^{m-r-2} &=& x_0^{m-r-2}+ 2(m-r-2) x_0^{2m-2r-4}y_0t+4\binom{m-r-2}{2} x_0^{3m-3r-6}y_0^2 t^2+\cdots.\\
\end{eqnarray*}
From $y^2x^{m-r-2}=g(x)$, it is easily seen that  $\ell_P$ intersects $\mathcal{C}$ in $P$ with multiplicity larger than $2$ if and only if either $x_0=0$ or 

$$h(x_0)^2+ 4(m-r-2) x_0^{m-r-2}y_0^2h(x_0)+4\binom{m-r-2}{2} x_0^{2m-2r-4}y_0^4=2x_0^{m-r}y_0^2 g^{\prime\prime}(x_0),$$

that is 
$$h(x_0)^2+ 4(m-r-2) g(x_0)h(x_0)+4\binom{m-r-2}{2} (g(x_0))^2=2x_0^{2}g(x_0) g^{\prime\prime}(x_0).$$

The above equation can be written as 
$$-r(r+2)a_0^2x_0^{2m}+\cdots=0,$$ 
and since $p\nmid r(r+2)$, it is satisfied by at most  $2m$ distinct  values $\eta_1,\ldots,\eta_{2m}$.  The corresponding points in $\mathcal{C}$ have tangent lines intersecting $\mathcal{C}$ with multiplicity at least $3$. Each $\eta_i$ gives rise to at most two points $\left(\eta_i,\pm \frac{\sqrt{g(\eta_i)}}{\eta_i^{m-r-2}}\right)$. This means that there are at most $4m$ tangent lines intersecting $\mathcal{C}_f$ with multiplicity larger than $2$. 

Summing up tangent lines $\ell_{P}$ intersecting $\mathcal{C}_f$ in less than $m-1$ distinct points are at most
$$ \underbrace{2m(4m-r-1)}_{\mathrm{two tangency points}}+\underbrace{m}_{\tiny {\begin{array}{c} x=x_0\\ g(x_0)=0\end{array}}}+\underbrace{4m}_{\tiny \begin{array}{c}\mathrm{multiplicity at } P\\ \mathrm{larger than 2}\end{array}}+\underbrace{2}_{\tiny \begin{array}{c} y=0,\\ \ell_{\infty} \end{array}}\leq 7m^2+3m +2$$
tangent lines of $\mathcal{C}_f$ intersecting the curve in less than $m-1$ points.
\end{proof}

Set 
\begin{equation}\label{Eq:Lambda}
\Lambda=\{\ell_P \ : \ \ell_P \text{ tangent at } P, |\ell_P\cap \mathcal{C}_f|<m-1\}\end{equation} 
and 
\begin{equation}\label{Eq:Gamma}
\Gamma = \bigcup _{\ell_P \in \Lambda}\{ Q \in \ell_P\}.
\end{equation}  

\begin{theorem}\label{Th:GaloisGroup}
Let $p$ be an odd prime number, $q$ be a power of $p$, $m$ be a positive integer, $r$ be another prime such that $m/2<r\leq m-2$, and suppose that $p\nmid r(r+2)(m-r-2)$. Consider  a  squarefree polynomial $g(x)\in  \mathbb F_q[x]$ of degree $m$ such that $g(0)\neq 0$. Let $L$ be the function field 
$\overline {\mathbb F_q}(x)[y]/(x^{m-r-2}y^2-g(x))$. Let $(a,b) \in \mathrm{AG}(2,q)\setminus (\mathcal{C}_f\cup \Gamma)$,  $F_{a,b}=\overline {\mathbb F_q}((y-b)/(x-a))$, and  $G_{a,b}$ be the Galois  group  of the Galois closure $M_{a,b}$ of the function field extension $L:F_{a,b}$. Then  $G_{a,b}\cong \mathcal{S}_m$. 
\end{theorem}

\begin{proof}
By simply using the equation for the sheaf of lines through $(a,b)$, i.e. $y=t(x-a)-b$, one obtains that $M$ is the splitting field of $H_{a,b}(x,t)$ from \eqref{Eq:g}. We already know that the Galois group of $H_{a,b}(x,t)$ is either $\mathcal A_n$ or $\mathcal S_n$ by Proposition \ref{Prop:1}.
If we can prove that $G_{a,b}$ has a transposition we are done.
Since the degree of the extension $L:F_{a,b}$ is larger than $1$,  by Hurwitz formula then there is a ramified place $P$ of $F_{a,b}$. The point $(a,b)$ was chosen in such a way that every line through $(a,b)$ that is tangent to $\mathcal C_f$ 
has multiplicity exactly $2$ at the point of tangency, and every other multiplicity is one. So 
Lemma \ref{orbits} yields that any decomposition group 
$D(R|P)$ has an orbit of lenght $2$ and  $m-2$ fixed points when acting 
on $\Hom_{F_{a,b}}(L:M_{a,b})$ 
(which can simply be identified with the roots of $H_{a,b}(x,t)$). Now, any element of $D(R|P)\leq G_{a,b}$ that acts non trivially has to be a transposition, which concludes the proof, showing that $G_{a,b}$ is not alternating.
\end{proof}

\begin{remark}
The number $\mathcal{C}(\mathbb{F}_q)$ of degree one places of a curve $\mathcal{C}$ can be slightly different from the number $R_q(\mathcal{C})$ of points of $\mathcal{C}$ in $\mathrm{PG}(2,q)$. In our case, since the unique singular points of the curves $y^2x^{m-r-2}=g(x)$ is at infinity and there are at most $m-2$ places of degree one centered at it, the difference $|\mathcal{C}(\mathbb{F}_q)-R_q(\mathcal{C})|$ is at most $m$. For the construction of complete $m$-arcs we need to estimate $|R_q(\mathcal{C})|$ instead of $|\mathcal{C}(\mathbb{F}_q)|$, although these two numbers are close. 
\end{remark}

We are now in a position to prove our main result.
Set $C=C(L:F)$ to be the constant obtained in Theorem \ref{thm:existence_tot_split} with $L=\mathbb F_q(x,y)$ (with $x^{m-r-2}y^2=g(x)$).

\begin{theorem}\label{Th:main}
Let $m\geq 5$ be a positive integer, $p$ be an odd prime number, $q$ be a power of $p$, $r$ be another prime such that $m/2<r\leq m-2$. Suppose that $p\nmid r(r+2)(m-r-2)$.
Consider   a squarefree polynomial $g(x)\in \mathbb F_q[x]$ of degree $m$ such that $g(0)\neq 0$.
Let $R_q(\mathcal C_f)$ be the set of points of the curve
$\mathcal C_f$ defined by $x^{m-r-2}y^2=g(x)$ in $\mathrm{PG}(2,q)$.
Then there exist an explicit constant $C$ (depending only on $m$) and a positive integer $k\leq (7m^2+3m+2)m$ such that if $q>C$ then 
$$A(\mathbb{F}_q,m)\leq |R_q(\mathcal C_f)|+k.$$
\end{theorem}
\begin{proof}
Let $\Lambda$ be as in \eqref{Eq:Lambda}. Enumerate the lines in $\Lambda$ by $\ell_1,\dots, \ell_b$, with $b\leq 7m^2+3m+2$ and consider the following process.
Start with the set $X_0=R_q(\mathcal C_f)$, for any $i=1,2,\dots b$ construct $X_i$ by taking the union of $X_{i-1}$ and a set $S_i$ of points of $\ell_i$ selected in such a way that there are not $m+1$ points aligned in $X_i=X_{i-1}\cup S_i$, but every point in $\ell_i$ is aligned with at least $m$-points in $X_i$.
Simply because all points of $\ell_i$ are aligned, it is immediate to see that $|S_i|\leq m$.
Set $Y=X_b$ and observe that $|Y|\leq |R_q(\mathcal C_f)|+k$, where $k\leq (7m^2+3m+2)m$ because of the previous construction and Proposition \ref{Prop:TangentLines}. We claim that $Y$ is an $m$-arc as long as $q$ is large enough.

Let $(a,b) \in AG(2,q)$. If $(a,b)$ lies in  $\ell_i$ for some $i$, then there is nothing to do: by construction it is aligned with other $m$ points of $Y$ but not with other $m+1$.
If $(a,b)$ does not lie on any of the $\ell_i$'s, then we consider the function field extension $L:F_{a,b}$, where $L=\overline{\mathbb F}_q(x,y)$ (and $x^{m-r-2}y^2=g(x)$) and $F_{a,b}=\overline{\mathbb F_q}((y-b)/(x-a))$. Use now Theorem \ref{Th:GaloisGroup} to deduce that the Galois closure $M$ of $L:F_{a,b}$ has Galois group $G=\mathcal S_n$. Let now $M'$ be the Galois closure of the extension $L':F_{a,b}'$, where $L'=\mathbb F_q(x,y)$ and $F_{a,b}'=\mathbb F_q((y-b)/(x-a))$. By standard Galois Theory, the geometric Galois group $G=\Gal(M:F_{a,b})$ naturally embeds in the arithmetic Galois group $A=\Gal(M':F_{a,b}')$ because $M$ and $M'$ are splitting field of $H_{a,b}(x,t)$ (recall its definition in \eqref{Eq:g}) respectively over $\overline{\mathbb F_q}(t)$ and $\mathbb F_q(t)$.
This forces immediately $G\cong A\cong \mathcal S_n$, so that the field of constants of $M'$  has to be $\mathbb F_q$.
Now apply Theorem \ref{thm:existence_tot_split} to $L':F_{a,b}'$ obtaining the existence of a totally split place as long as $q>C$, and in turn a specialization $t_0$ such that $H_{a,b}(x,t_0)$  is totally split over $\mathbb F_q$. By the definition of  $H_{a,b}(x,t_0)$ this is equivalent to give a slope $t_0$ of a line through $(a,b)$ that intersects the curve $\mathcal C_f$ in $m$ distinct points.
The result follows.
\end{proof}

Whenever the characteristic is larger than the degree of the curve, we get an unconditional result, also with a nice explicit constant.
\begin{theorem}\label{thm:chargreaterm}
Let $m$ be an integer such that $m\geq 8$.
Let $p>m$ be an odd prime number and $q$ be a power of $p$,  $g(x)\in \mathbb F_q[x]$ be a squarefree polynomial of degree $m$ such that $g(0)\neq 0$. Denote by  $R_q(\mathcal C_f)$ be the set of  points of the curve
$\mathcal C_f$ defined by $x^{m-r-2}y^2=g(x)$ in $\mathrm{PG}(2,q)$, for $r$ a prime chosen in the interval $]m/2,m-2]$.
Then there exists a positive integer $k\leq (7m^2+3m+2)m$ such that if $q>9(1+ m/2+m)^2(m!)^2$ then  $$A(\mathbb{F}_q,m)\leq |R_q(\mathcal C_f)|+k.$$
\end{theorem}
\begin{proof}
Using \ref{Remark:Bertrand} we can immediately drop the hypothesis $p\nmid r(r+2)(m-r-2)$. Using now Corollary \ref{cor:constant} we get that $q>9(g_L+m)^2(m!)^2$, as the $F_{a,b}$ in the proof of Theorem \ref{Th:main} have always genus $0$. To obtain $g_L$ we now use \cite[Corollary 3.7.4]{17} and 
\begin{eqnarray*}
g_L&=&1+2\cdot(-1)+\frac{1}{2}\left(\sum_{P \in \mathbb{P}_{\overline{\mathbb{F}}_q(x)}}2-\gcd(2,v_P(g(x)/x^{m-r-2}))\right)\\
&=&-1+\frac{1}{2}\Big(m+1+(m-r \!\!\!\!\mod {2})\Big)=\left\lfloor \frac{m}{2}\right\rfloor.
\end{eqnarray*}
\end{proof}

In the small characteristic regime, an analogous result holds (here we do not compute the constants explicitly as they do depend on complicated error terms)
\begin{theorem}\label{cor:size_of_arc}
Let $p$ be an odd prime number greater than or equal to $5$, $q$ be a power of $p$,  $r$ be a prime number, $m$ be a positive integer,  $g(x)\in \mathbb F_q[x]$ be a squarefree polynomial of degree $m$ such that $g(0)\neq 0$.
 Let $R_q(\mathcal{C})$ be the set of  points of the curve
$\mathcal C$ defined by $x^{m-r-2}y^2=g(x)$ in $\mathrm{PG}(2,q)$. Then $r$ can be chosen in such a way that there is an absolute constant $C$ such that for every  $m>C$, there exist an explicit constant $D$ (depending only on $m$) and a positive integer $k\leq (7m^2+3m+2)m$ such that if $q>D$ then 
$$A(\mathbb{F}_q,m)\leq |R_q(\mathcal C_f)|+k.$$
\end{theorem}
\begin{proof}
Here it is only required to show that such prime $r\in]m/2,m-2]$ always exists and it is such that $p\nmid r(r+2)(m-r-2)$. This is a consequence of Dirichlet density theorem, we sketch the proof for the reader's convenience.
Suppose that there is no constant $C$ such that the theorem is true. Then there is an infinite subsequence of $m$'s such that all primes in $]m/2,m-2]$ are such that $p\nmid r(r+2)(m-r-2)$. But then, by the infinite pigeonhole principle, one can select a subsequence of $m$'s in such a way that all primes in $]m/2,m-2]$ are of the form $a+kp$. For $m\rightarrow \infty$, this would give a density of primes on the arithmetic progression $a+kp$ at least larger than $1/4$. But this is impossible as the density of primes equidistributes over the $p-1$ different reductions modulo $p$ and $1/p<1/4$ by the assumption on $p$.
\end{proof}

\section{Constructions of $m$-arcs of size less than $q$}\label{Section:5}

Using our construction, we are now able to provide infinitely many new examples of $m$-arcs that have size much lower than $q$. Here we present both an asymptotic result and an explicit construction of small complete $m$-arcs.

\begin{lemma}\label{lemma:subsequence}
Let $g$ be a positive integer. Let $x_1,\dots x_g\in \mathbb R$. 
Consider the sequence
\[ a_m=\sum^g_{i=1}\cos(mx_i).\]
Then there is a sequence $\{m_k\}_{k\in \mathbb N}$ such that  $a_{m_k}\rightarrow g$ when $k\rightarrow \infty$.
\end{lemma}
\begin{proof}
First observe that the $x_i$'s that are rational multiples of $\pi$ can be dropped at the price of selecting a subsequence of the form $2vm$, where $v$ is the least common multiple of the denominators $v_i$ when one writes $x_i=\pi u_i/v_i$. So now we can redefine the other $x_i$'s by $2vx_i$, and therefore we are allowed to consider the  problem where all the $x_i$'s are not rational multiples of $\pi$.

Now that we are sure that all the $x_i/\pi$'s are irrational, by Dirichlet theorem on Symultaneous Diophantine Approximation \cite{Cassels}, for any $i\in\{1,\dots g\}$ there is a sequence 
$p_k^{(i)}/q_k$ that approximates $x_i/\pi$ with error term exactly $1/q_k^{1+1/g}$, i.e.
\[x_i/\pi=p_k^{(i)}/q_k+\varepsilon_k \quad \text{with} \quad |\varepsilon_k|\leq 1/q_k^{1+1/g}.\]

Select the sequence $m_k=2q_k$. If we can show that $\cos (2q_kx_i)\rightarrow 1 $ we are done. But this is obvious, as 
\begin{eqnarray*}
\cos(2q_kx_i)&=&\cos(2\pi q_k(x_i/\pi))=\cos(2\pi q_k (p_k^{(i)}/q_k +\varepsilon_k))\\
&=&\cos(2\pi p_k^{(i)}+2\pi q_k\varepsilon_k)=\cos(2\pi q_k\varepsilon_k).
\end{eqnarray*}
Since the sequence $q_k\varepsilon_k\rightarrow 0$ and the cosinus is continuous,  it follows $\cos (2q_kx_i)\rightarrow 1$ and therefore the final claim.
\end{proof}
The next proposition is used to show that the lower boundary of the Hasse interval is always asymptotically reached.
\begin{proposition}\label{prop:asymptotic_minimality}
Let $F$ be a function field and let  $\mathcal P_m$ be the set of degree $1$ rational places of $\mathbb F_{q^m}F$.
Then there exists a subsequence of extensions $\mathbb F_{q^{m_k}}$ such that
\[|\mathcal P_{m_k}|-(q^{m_k}+1)\sim -2gq^{m_k/2}.\]
\end{proposition}
\begin{proof}
Let $L_F(T)=\prod_{i=1}^{2g} (1-\alpha_iT)$ be the $L$-polynomial of $F$, with $\alpha_i\in\mathbb C$ such that $|\alpha_i|=q^{1/2}$ and $\overline \alpha_i=\alpha_{g+i}$ for $i\in\{0,\dots g-1\}$.
By the standard theory of function fields over finite fields we have
\[|\mathcal P_m|-(q^{m}+1)=-\sum^{2g}_{i=1}\alpha_i^m.\]
By pairing $\alpha_i^m$ with its conjugate and using the equality $|\alpha_i^m|=q^{m/2}$ we get
\[|\mathcal P_m|-(q^{m}+1)=-2q^{m/2}\left(\sum^g_{i=1}\cos(m x_i)\right).\]
Using now Lemma \ref{lemma:subsequence} we get that there exists a subsequence $m_k$ such that
\[|\mathcal P_{m_k}|-(q^{m_k}+1)=-2q^{m_k/2}\left(\sum^g_{i=1}\cos(m_k x_i)\right)\sim -2gq^{m_k/2}.\]
\end{proof}

We are now in position to provide our asymptotic result on complete $m$-arcs. 
\begin{theorem}\label{Th:Asymptotic}
Let $m\geq 8$ be a positive integer, $K$ be a  finite field of characteristic greater than $3$.
For all $m$'s but a finite number, there exists an absolute constant $C=C(m)$ such that there is a sequence of integers $n_k$ such that 
\[\limsup_{k\rightarrow \infty} \frac{q^{n_k}-|A(\mathbb F_{q^{n_k}},m)|}{C\sqrt{q^{n_k}}}\geq 1.\]
In other words, for all but finitely many $m$'s, if the base field is large enough, there are always complete $m$-arcs $A$ such that $|A|\leq q^{n_k}-C(m)\sqrt{q^{n_k}}$.
\end{theorem}
\begin{proof}
We have all the ingredients to see this immediately. If the characteristic of $K$ is greater than $m$ (resp. if the characteristic of $K$ is smaller than $m$), fix the curve $\mathcal C$ defined in  Theorem \ref{thm:chargreaterm} (resp. in Theorem \ref{cor:size_of_arc}) and use Proposition \ref{prop:asymptotic_minimality} on $\mathcal C$ to obtain a sequence of integers $n_k$ such that $\mathcal C(\mathbb F_{q^{n_k}})-(q^{n_k}+1)\sim-2gq^{n_k/2} $. Now use Theorem  \ref{thm:chargreaterm} (resp. Theorem \ref{cor:size_of_arc}) to construct an arc of cardinality $|R_q(\mathcal C)|+s$, where $s$ is absolutely bounded by a constant depending only on $m$. The claim follows immediately.
\end{proof}

Explicit constructions of small complete $m$-arcs are described in the following.

\begin{proposition}\label{prop:Ex1}
Let $q$ be a prime power. Let $m$ be an odd divisor of $q+1$ and $m/2<r\leq m-2 $ a prime. Then there exists $g(x)\in \mathbb F_q[x]$ such that the curve $y^2x^{m-r-2}=g(x)$ has at most $$q^2-(m+1) q+3m$$ points in $\mathrm{PG}(2,q^2)$.
\end{proposition}
\proof
Consider the curve $\mathcal{E}: y^2=x^m+1$. It can be seen that  $\mathcal{E}$ is a quotient of the Hermitian curve $y^{q+1}=x^{q+1}+1$ and therefore it is maximal over $\mathbb{F}_{q^2}$; see also \cite[Theorem 1]{Tafazolian}. Thus, the number of the  $\mathbb{F}_{q^2}$-rational places of $\mathcal{E}$ is $q^2+1+2\lfloor m/2\rfloor q=q^2+(m-1)q+1$ and it possesses at least 
$$N_{q^2}=q^2+1+(m-1) q-(m-2)-2-m$$
 points $(x_0,y_0)\in \mathrm{AG}(2,q^2)$ with  $x_0(x_0^m+1)\neq 0$. Let $A$ be the set of such points, so that $|A|\geq N_{q^2}$.
Consider now the curve $\mathcal{E}^{\prime}: (yx^{(m-r-2)/2})^2=x^m+1$. Let $B$ be the set of  points of $\mathcal{E}^{\prime}$ in $\mathrm{AG}(2,q^2)$ such that $x_0\neq 0$ and such that $x_0^m+1\neq 0$. Since the map $(x_0,y_0)\mapsto (x_0,y_0/x_0^{(m-r-2)/2})$ gives a bijection between the rational points of $\mathcal{E}$ and the ones of $\mathcal E^{\prime}$, we have that $|B|=|A|\geq N_{q^2}$. Finally, consider the curve 
$\mathcal{C}: (yx^{(m-r-2)/2})^2=\xi(x^m+1)$, where $\xi$ is a nonsquare in $\mathbb{F}_{q^2}$. Let $Y$ be the set of  points of $\mathcal C$ in $\mathrm{AG}(2,q^2)$ such that  $x_0(x_0^m+1)\neq 0$.
We now want to estimate $|Y|$.
To do that, consider the set $S$ consisting of the projection on the $x$-coordinate of $B$ and $Y$. 
Clearly $S$ contains at most $q^2$ elements and all of them satisfy   $x_0(x_0^m-1)\neq 0$. On the other hand, the projection on the $x$ coordinate of $B$ has at least $N_{q^2}/2$ affine points with  $x_0(x_0^m+1)\neq 0$. 
So we have
\[N_{q^2}/2+|Y|/2\leq |B|+|Y|=|S|\leq q^2\]
and therefore
\[|Y|\leq 2q^2 -N_{q^2}=q^2-(m-1) q+2m-1.\]
Adding now at most all the points such that $x^m+1=0$ (which are at most $m$) and one point at infinity we get that $\mathcal{C}$ has at most 
$$ q^2-(m-1) q+3m$$ 
 points in $\mathrm{PG}(2,q^2)$. 
\endproof

\begin{proposition}\label{prop:Ex2}
Let $q$ be a prime power. Let $m$ be an even integer such that $2m$ divides $q+1$ or $q-1-m$ and let $m/2<r\leq m-2 $ a prime. Then there exists $g(x)\in\mathbb F_q[x]$ such that the curve $y^2x^{m-r-2}=g(x)$ has at most $$q^2-m q+3m+1$$
points in $\mathrm{PG}(2,q^2)$.
\end{proposition}
\proof
Consider the curve $\mathcal{E}: y^2=x^{m+1}+x$. By our assumptions, \cite[Theorem 3.1]{TT} yields that  $\mathcal{E}$ is $\mathbb{F}_{q^2}$-maximal. Since there are no affine singular points and at most $m-1$ places are centered at $P_{\infty}$, there are at least $N_{q^2}=q^2+m q-(m-1)-(m+1)$ points $(x_0,y_0)\in \mathrm{AG}(2,q^2)$ such that $x_0^{m+1}+x_0\neq 0$. Therefore, if $\xi$ is a non-square of $\mathbb F_q$, by the same argument as Proposition \ref{prop:Ex1} the curve $\mathcal{E}^{\prime}: y^2=\xi(x^{m+1}+x)$ has at most $2q^2-N_{q^2}=q^2-mq+2m$ points $(x_0,y_0)\in \mathrm{AG}(2,q^2)$ such that $x_0^{m+1}+x_0\neq 0$.

Consider now the curve  
$$\mathcal{C}: (yx^{(m-r-1)/2})^2/x=\xi(x^{m}+1).$$
Since the map $(x_0,y_0)\mapsto (x_0,y_0/x_0^{(m-r-1)/2})$ is a bijection between the affine $\mathbb F_{q^2}$ rational points of $\mathcal{E}^{\prime}$ and the ones of $\mathcal{C}$, that is also well defined on the set of affine points $(x_0,y_0)$ such that $x_0^{m+1}+x_0\neq 0$, the number of  points of $\mathcal{C}$ in $\mathrm{PG}(2,q^2)$ is at most 
$$q^2-mq+2m+m+1=q^2-mq+3m+1.$$
\endproof

\begin{remark}
Note that maximal curves $\mathcal{E}$ described in the proofs of Propositions \ref{prop:Ex1} and \ref{prop:Ex2} are minimal over $\mathbb{F}_{q^4}$ and they can be used to produce small complete $m$-arcs in $\mathrm{PG}(2,q^4)$.
\end{remark}

Finally, we collect the last two explicit constructions of small complete arcs in $\mathrm{PG}(2,q^2)$ in the following theorem.
\begin{theorem}\label{Th:minimal}
Let $m\geq 8$ be a positive integer, $p$ be an odd prime number, and $q$ be a power of $p$. Suppose that there exists another prime $r$ such that $m/2<r\leq m-2$ and  $p\nmid r(r+2)(m-r-2)$.

Suppose that one of the following condition is satisfied.

\begin{enumerate}
    \item $m$ is an odd divisor of $q+1$;
    \item $m$ is even and  $2m$ divides $q+1$ or $q-1-m$.
\end{enumerate}
Then there exists a complete $m$-arc in $\mathrm{PG}(2,q^2)$ of size at most 
$$q^2-2\lfloor {m}/{2}\rfloor q+7m^3+3m^2+5m+1.$$

\end{theorem}
\proof
 It follows from Propositions \ref{prop:Ex1} and \ref{prop:Ex2} and Theorem \ref{Th:main}.
\endproof

The machinery described in this paper provides  asymptotic bounds for complete $m$-arcs in projective planes. Clearly, one can expect that  curves of higher genus can produce arcs of smaller size in $\mathrm{PG}(2,q)$, although always of size $q+O(m)\sqrt{q}$. Here, we want to point out that considering different polynomials $g(x)\in \mathbb{F}_q[x]$ in Section \ref{Section:y2} yields a large spectrum of possible sizes of complete $m$-arcs.

Tables \ref{Table:1} and \ref{Table:2} collect  results about the number of affine $\mathbb{F}_q$-rational points of the curve $y^2x^{m-r-2}-g(x)=0$, for specific values of $m$ and $r$ and $g(x)=x^m+\alpha x^2+\alpha x+\beta$, $\alpha\neq\beta\in \mathbb{F}_q\setminus\{0\}$. Such a choice for $g(x)$ covers only a thin part of the whole search space, but already provides evidence that the number of $\mathbb F_q$-rational points of the family of curves $y^2x^{m-r-2}-g(x)=0$ covers a wide spectrum within the Hasse interval. As a byproduct of this, we obtain a large variety of $m$-arcs that can be explicitly constructed, and that have size smaller than $q$.

\begin{table}
\caption{Number $N_{8,5,q}$ of affine $\mathbb{F}_q$-rational points of  $y^2x-(x^8+\alpha x^2+\alpha x+\beta)=0$, $\alpha\neq\beta\in \mathbb{F}_q^*$}
\label{Table:1}
\centering
\begin{tabular}{|c|l|}
\hline
$q$&$N_{8,5,q}$\\
\hline
\hline
$11$& $\begin{array}{l} 4, 5, 6, 7, 8, 9, 10, 11, 12, 13, 14, 15, 16\end{array}$\\
\hline
$13$& $\begin{array}{l}4, 5, 6, 8, 9, 10, 11, 12, 13, 14, 15, 16, 17, 18, 19\end{array}$\\
\hline
$17$& $\begin{array}{l}7, 9, 10, 11, 12, 13, 14, 15, 16, 17, 18, 19, 20, 21, 22, 23, 24, 26, 28\end{array}$\\
\hline
$19$& $\begin{array}{l}6, 8, 9, 10, 11, 12, 13, 14, 15, 16, 17, 18, 19, 20, 21, 22, 23, 24, 26, 27, 29\end{array}$\\
\hline
$23$& $\begin{array}{l}11, 12, 13, 14, 15, 16, 17, 18, 19, 20, 21, 22, 23, 24, 25, 26, 27, 28, 29, 30, 31,\\ 32, 33, 
34, 36\end{array}$\\
\hline
$25$& $\begin{array}{l}10, 12, 13, 14, 16, 17, 18, 19, 20, 21, 22, 23, 24, 25, 26, 27, 28, 29, 30, 31, 32,\\ 33, 34, 
35, 36, 38\end{array}$\\
\hline
$27$& $\begin{array}{l}12, 15, 16, 18, 19, 21, 22, 24, 25, 27, 28, 30, 31, 33, 34, 36, 37, 40\end{array}$\\
\hline
$29$& $\begin{array}{l}15, 16, 17, 18, 19, 20, 21, 22, 23, 24, 25, 26, 27, 28, 29, 30, 31, 32, 33, 34, 35,\\ 36, 37, 
38, 40, 41\end{array}$\\
\hline
$31$& $\begin{array}{l}14, 16, 17, 18, 19, 20, 21, 22, 23, 24, 25, 26, 27, 28, 29, 30, 31, 32, 33, 34, 35,\\ 36, 37, 
38, 39, 40, 41, 42, 43, 44, 46\end{array}$\\
\hline
$37$& $\begin{array}{l}18, 22, 23, 24, 25, 26, 27, 28, 29, 30, 31, 32, 33, 34, 35, 36, 37, 38, 39, 40, 41,\\ 42, 43, 
44, 45, 46, 47, 48, 49, 50, 51, 53\end{array}$\\
\hline
$41$& $\begin{array}{l}22, 23, 24, 25, 26, 27, 28, 29, 30, 31, 32, 33, 34, 35, 36, 37, 38, 39, 40, 41, 42,\\ 43, 44, 
45, 46, 47, 48, 49, 50, 51, 52, 53, 54, 55, 56, 58\end{array}$\\
\hline
$43$& $\begin{array}{l}18, 24, 25, 26, 27, 28, 29, 30, 31, 32, 33, 34, 35, 36, 37, 38, 39, 40, 41, 42, 43,\\ 44, 45, 
46, 47, 48, 49, 50, 51, 52, 53, 54, 55, 56, 58, 59, 60, 64\end{array}$\\
\hline
$47$& $\begin{array}{l}23, 24, 27, 28, 29, 30, 31, 32, 33, 34, 35, 36, 37, 38, 39, 40, 41, 42, 43, 44,45,\\  46, 47, 
48, 49, 50, 51, 52, 53, 54, 55, 56, 57, 58, 59, 60, 61, 62, 64, 65, 66\end{array}$\\
\hline
$49$& $\begin{array}{l} 30, 31, 32, 33, 34, 35, 36, 37, 38, 39, 40, 41, 42, 43, 44, 45, 46, 47, 48, 49, 50,\\ 51, 52, 
53, 54, 55, 56, 57, 58, 59, 60, 61, 62, 63, 64, 65, 66, 68\end{array}$\\
\hline
\end{tabular}
\end{table}

\begin{table}
\caption{Number $N_{11,7,q}$ of affine $\mathbb{F}_q$-rational points of $y^2x^{2}-(x^{11}+\alpha x^2+\alpha x+\beta)=0$, $\alpha\neq\beta\in \mathbb{F}_q^*$}
\label{Table:2}
\centering
\begin{tabular}{|c|l|}
\hline
$q$&$N_{11,7,q}$\\
\hline
\hline
$11$& $\begin{array}{l} 1, 8, 10, 12, 19\end{array}$\\
\hline
$13$& $\begin{array}{l}4, 5, 7, 8, 9, 10, 11, 12, 13, 14, 15, 16, 17, 18\end{array}$\\
\hline
$17$& $\begin{array}{l}6, 8, 9, 10, 11, 12, 13, 14, 15, 16, 17, 18, 19, 20, 21, 22, 23, 24\end{array}$\\
\hline
$19$& $\begin{array}{l}6, 8, 9, 10, 12, 13, 14, 15, 16, 17, 18, 19, 20, 21, 22, 23, 24, 25, 26, 28, 30\end{array}$\\
\hline
$23$& $\begin{array}{l}8, 9, 10, 12, 13, 14, 15, 16, 17, 18, 19, 20, 21, 22, 23, 24, 25, 26, 27, 28, 29, 30, 31,\\ 32,
33, 35, 36\end{array}$\\
\hline
$25$& $\begin{array}{l}13, 14, 15, 16, 17, 18, 19, 20, 21, 22, 23, 24, 25, 26, 27, 28, 29, 30, 31, 32, 33, 34,\\ 35, 
36 \end{array}$\\
\hline
$27$& $\begin{array}{l}13, 14, 15, 16, 17, 18, 19, 20, 21, 22, 23, 24, 25, 26, 27, 28, 29, 30, 31, 32, 33, 34,\\ 35, 
36, 37, 38\end{array}$\\
\hline
$29$& $\begin{array}{l}12, 14, 15, 16, 17, 18, 19, 20, 21, 22, 23, 24, 25, 26, 27, 28, 29, 30, 31, 32, 33, 34,\\ 35, 
36, 37, 38, 39, 40, 41, 42, 43\end{array}$\\
\hline
$31$& $\begin{array}{l}15, 16, 17, 19, 20, 21, 22, 23, 24, 25, 26, 27, 28, 29, 30, 31, 32, 33, 34, 35, 36, 37,\\ 38, 
39, 40, 41, 42, 43, 44, 51\end{array}$\\
\hline
$37$& $\begin{array}{l} 14, 18, 20, 21, 22, 23, 24, 25, 26, 27, 28, 29, 30, 31, 32, 33, 34, 35, 36, 37, 38, 39,\\ 40, 
41, 42, 43, 44, 45, 46, 47, 48, 49, 50, 51, 54, 56 \end{array}$\\
\hline
$41$& $\begin{array}{l}23, 26, 27, 28, 29, 30, 31, 32, 33, 34, 35, 36, 37, 38, 39, 40, 41, 42, 43, 44, 45, 46,\\ 47, 
48, 49, 50, 51, 52, 53, 54, 55, 56, 57, 59, 60, 64\end{array}$\\
\hline
$43$& $\begin{array}{l}21, 24, 26, 27, 28, 29, 30, 31, 32, 33, 34, 35, 36, 37, 38, 39, 40, 41, 42, 43, 44, 45,\\ 46, 
47, 48, 49, 50, 51, 52, 53, 54, 55, 56, 57, 58, 60, 61\end{array}$\\
\hline
$47$& $\begin{array}{l}24, 25, 26, 27, 28, 29, 30, 31, 32, 33, 34, 35, 36, 37, 38, 39, 40, 41, 42, 43, 44, 45,\\ 46, 
47, 48, 49, 50, 51, 52, 53, 54, 55, 56, 57, 58, 59, 60, 61, 62, 63, 64, 68\end{array}$\\
\hline
$49$& $\begin{array}{l} 22, 30, 31, 32, 33, 34, 35, 36, 37, 38, 39, 40, 41, 42, 43, 44, 45, 46, 47, 48, 49, 50,\\ 51, 
52, 53, 54, 55, 56, 57, 58, 59, 60, 61, 62, 63, 64, 65, 68\end{array}$\\
\hline
\end{tabular}
\end{table}

\section*{Acknowledgments}
The research of D. Bartoli  was partially supported  by the Italian National Group for Algebraic and Geometric Structures and their Applications (GNSAGA - INdAM). Part of this work was done while the first author was visiting University of South Florida.

\end{document}